\documentclass[11pt]{amsart}
\usepackage{
amssymb,
amsmath,mathdots}

\usepackage{arydshln}
\usepackage[bookmarks,colorlinks,pagebackref]{hyperref} 

\usepackage{graphicx}

\usepackage{hyperref}

\usepackage[latin1]{inputenc}

\usepackage{mathpazo}
\usepackage[scaled=.95]{helvet}
\usepackage{courier}

\usepackage[all,cmtip,2cell]{xy}

\numberwithin{equation}{section}
\theoremstyle{plain} % style de mise en pages plain: normal (
\newtheorem{proposition}{Proposition}[section]  % Proposition 
\newtheorem{lemma}[proposition]{Lemma}
\newtheorem{corollary}[proposition]{Corollary} % idem
\newtheorem{theorem}[proposition]{Theorem} % idem

\theoremstyle{definition} % style de mise en page d\UTF{00E9}finition

\newtheorem{remark}[proposition]{Remark} % Remarque 

\newtheorem{example}[proposition]{Example} % Exemple 

\newtheorem{question}[proposition]{Question}
\newtheorem{observation}[proposition]{Observation}

\newtheorem{notation}[proposition]{Notation}

\newtheorem{construction}[proposition]{Construction}
\newtheorem{notation and recalls}[proposition]{Notations and Recalls}

\newcommand\Tor{\operatorname{Tor}}
\newcommand\Hom{\operatorname{Hom}}

\newcommand\Ext{\operatorname{Ext}}
\newcommand\Rad{\operatorname{Rad}}

\newcommand\reg{\operatorname{reg}}
\newcommand\Coker{\operatorname{Coker}}
\newcommand\im{\operatorname{Im}}

\newcommand\grade{\operatorname{grade}}

\newcommand\depth{\operatorname{depth}}

\author[T.~Hibi, P.~Schenzel]{Takayuki Hibi and Peter Schenzel}
\title[Locally linear resolutions]{Monomial curves and locally linear resolution}

\address{Department of Pure and Applied Mathematics, Graduate School of Information Science and Technology, The University of Osaka, Suita, Osaka 565-0871, Japan}
\email{hibi@math.sci.osaka-u.ac.jp}
\address{Martin-Luther-Universit\"at Halle-Wittenberg,
Institut f\"ur Informatik, D - 06 099 Halle (Saale), Germany}
\email{peter.schenzel@informatik.uni-halle.de}

\date{\today}

\begin{document}

\begin{abstract} 
For four elements of a Noetherian ring we construct complexes of free modules of length three (resp. five) 
by an explicit description of the homomorphisms of the free modules. We provide exactness criteria 
for them. As an application we use these results in order to describe explicit the minimal free resolution 
of the Hartshorne--Rao module of a monomial curve lying on a smooth quadric. Also it provides an 
example of linearly generated module with syzygies of arbitrary high degree. 
\end{abstract}


\maketitle

\section{Introduction}
Let $A=K[x_1, \ldots, x_n]$ denote the polynomial ring in $n$ variables over a field $K$ with each $\deg x_i = 1$.  Let $M = \bigoplus_{i=0}^{\infty} M_i$ be a finitely generated positively graded module over $A$ and
\[
	\label{mgfr}
	0 \to F_p \to F_{p-1}
\to  \cdots  \to F_1 \to  F_0 \to  M  \to  0  \tag{$\star$}
\]
the minimal graded free resolution of $M$ with 
\[
	\label{betti}
	F_i = \oplus_{j} A(-j)^{\beta_{ij}} = \oplus_{j=1}^{\beta_i} A(-d_{ij}). 
\]
We say that (\ref{mgfr}) is {\em locally linear} if there are integers $d < e$, $1 < s < p-1$ for which
\[
F_i = \begin{cases}
	A(-d-i)^{\beta_{i}} & \text{ if } 0 \leq i < s, \\
	 A(-d-s)^{\beta_{s,\,d+s}} \oplus A(-e-s)^{\beta_{s,\,e+s}} & \text{ if } i =s,\\
	  A(-e-i)^{\beta_{i}} & \text{ if }  s < i \leq p. 
	\end{cases}
\]
When (\ref{mgfr}) is locally linear, we call $g = e - d > 0$ the {\em gap} of $M$. %(\ref{mgfr}).  
At first glance, it seems unclear if there exists a locally linear resolution with an arbitrary large gap.  

\begin{question}
	Does there exist a locally linear resolution of a graded module whose gap is arbitrary large? 
	
	Especially, does there exist a locally linear resolution of a monomial ideal (\cite{HHgtm260}) whose gap is arbitrary large?  
\end{question} 

In the present paper, by using the defining ideal of a certain  monomial curve, we construct a finitely generated positively graded module over $A$ whose minimal graded free resolution is locally linear with an arbitrary large gap.      

Let $0 < b \leq a-2$ be integers with $\gcd(a,b) = 1$ and $d = a+b$.  Let $C \subset \mathbb{P}_\Bbbk^3$ denote the monomial curve given as the image 
\[
(s:t) \in \mathbb{P}_\Bbbk^1 \mapsto (s^d:s^at^b:s^bt^a:t^d) \in \mathbb{P}_\Bbbk^3.
\]
Let $A = K[x_0,x_1,x_2,x_3]$ and $I_C \subset A$ the defining ideal of $C$.  We investigate the Hartshorne--Rao module ${HR}(C) = H^1_{A_+}(A/I_C)$ of $C$.  

\begin{theorem} \label{thm} 
	With the previous notation we construct an $(a-b-1)\times (2(a-b+1))$ matrix $\mathcal{D}$ such that: 
	\begin{itemize}
		\item[(a)] $HR(C) (b)\cong A^{a-b-1}/\im \mathcal{D}$.
		\item[(b)] 
		The minimal free resolution of $HR(C)$ is given by
		\begin{gather*}
			0 \to A^{a-b-1}(-a-2) \to A^{2(a-b)}(-a-1) \to A^{a-b+1}(-a) \oplus A^{a-b+1}(-b-2) \\
			\to A^{2(a-b)}(-b-1) \to A^{a-b-1}(-b) \to  HR(C)  \to 0.
		\end{gather*}
		\item[(c)] $\Hom_\Bbbk(HR(C), \Bbbk) \cong HR(C)(a+b-2)$.
	\end{itemize}
\end{theorem}

In particular, as a corollary of (ii) of Theorem \ref{thm}, it follows that

\begin{corollary}
	\label{main}
	There exists a locally linear resolution whose gap is arbitrary large. 
\end{corollary}

The present paper is organized as follows:  First we fix a sequence of elements $\underline{f} = (f_1,f_2,f_3,f_4)$ over a commutative Noetherian ring $A$. Second, based on some elementary linear algebra, we  define matrices $\mathcal{A}, \mathcal{A}', \mathcal{B}, \mathcal{C}, \mathcal{D}$ over $A$ of appropriate size and prove some annihilation between them   in Section $3$. As a consequence 
they provide  several complexes of free $A$-modules, among them 
\begin{gather*}
	C^{\bullet}: 0 \to A^{a-1} \stackrel{\mathcal{C}}{\longrightarrow} A^{2a}   \stackrel{\mathcal{A}}{\longrightarrow} A^{2(a+1)} 
	\stackrel{\mathcal{B}}{\longrightarrow} A^{2a}  \stackrel{\mathcal{D}}{\longrightarrow} A^{a-1} \\
	C^{\bullet}_1: 0 \to A^{a-1} \stackrel{\mathcal{C}}{\longrightarrow} A^{2a} \stackrel{\mathcal{A}'}{\longrightarrow} A^{a+1}.
\end{gather*}
Under some weak conditions these complexes are exact. 
Moreover,  in Section $5$, as an application of the previous constructions we study the monomial curve on the smooth quadric given by $Q = x_0x_3-x_1x_2$ and prove Theorem \ref{thm}. It would be a challenge to extend some of these results to larger sequences and arbitrary monomial curves in projective $3$-space.

\section{Notation and Preliminaries}
In the sequel let $A$ denote a commutative Noetherian ring. We shall fix $\underline{f} = f_1,f_2,f_3,f_4$ 
a sequence of four elements of $A$.  With these elements we shall define several matrices.

\begin{construction} \label{cons-1} (A):
	For a given positive  integer $a $ we consider the following four $a\times (a+1)$-matrices: 
	\[
	\mathcal{A}_{1,1} = 
	\begin{pmatrix}
		f_1 & -f_2 &  &  &  \\
		& f_1 & -f_2&  &  \\
		&         &     \ddots  &  \ddots &\\
		&          &       &              f_1 & -f_2
	\end{pmatrix} 
	\mathcal{A}_{1,2} = 
	\begin{pmatrix}
		f_3^{a-1} & 0 &  0 & 0& \ldots   & 0  &0 \\
		f_3^{a-2}f_4 & f_1f_3^{a-2}&  0 & 0& \ldots &0  & 0 \\
		f_3^{a-3}f_4^2 & f_1 f_3^{a-3}f_4 & f_1^2f_3^{a-3}& 0&\ldots &0  &0 \\
		\vdots & \vdots & \vdots & \vdots& \ddots & \vdots & \vdots\\
		f_4^{a-1} & f_1f_4^{a-2} & f_1^2f_4^{a-3} &f_1 ^3f_4^{a-4} & \ldots & f_1^{a-1} & 0
	\end{pmatrix} 
	\]
	\[
	\mathcal{A}_{2,1} = 
	\begin{pmatrix}
		&   &  & -f_4 &  f_3\\
		&   & -f_4&  f_3&  \\
		&       \reflectbox{$\ddots$}    &      \reflectbox{$\ddots$}   & &\\
		-f_4&    f_3     &       &      & 
	\end{pmatrix} 
	\mathcal{A}_{2,2} = 
	\begin{pmatrix}
		0 & 0 & \ldots& 0 & 0& 0& f_1^{a-1} \\
		0 & 0 & \ldots & 0 & 0& f_1^{a-2}f_3& f_1^{a-2} f_2 \\
		0 & 0 &\ldots & 0 & f_1^{a-3}f_3^2& f_1^{a-3}f_2f_3 & f_1^{a-3}f_2^2 \\
		\vdots & \vdots & \reflectbox{$\ddots$}  & \vdots & \vdots & \vdots & \vdots \\
		0 & f_3^{a-1}& \ldots&f_2^{a-4}  f_3^3 & f_2^{a-3}f_3^2& f_2^{a-2}f_3 & f_2^{a-1}
	\end{pmatrix}
	\]
	In a similar way we define the following four $a \times (a+1)$ matrices:
	\[
	\mathcal{B}_{1,1} = 
	\begin{pmatrix}
		f_3^{a-1} & 0 &  0 & 0& \ldots   & 0  &0 \\
		f_2f_3^{a-2} & f_1f_3^{a-2}&  0 & 0& \ldots &0  & 0 \\
		f_2^2f_3^{a-3} & f_1 f_2f_3^{a-3} & f_1^2f_3^{a-3}& 0&\ldots &0  &0 \\
		\vdots & \vdots & \vdots & \vdots& \ddots & \vdots & \vdots\\
		f_2^{a-1} & f_1f_2^{a-2} & f_1^2f_2^{a-3} &f_1 ^3f_2^{a-4} & \ldots & f_1^{a-1} & 0
	\end{pmatrix} 
	\mathcal{B}_{1,2} = 
	\begin{pmatrix}
		-f_1 & f_4 &  &  &  \\
		& -f_1 & f_4&  &  \\
		&         &     \ddots  &  \ddots &\\
		&          &       &            -f_1 & f_4
	\end{pmatrix} 
	\]
	\[
	\mathcal{B}_{2,1} = 
	\begin{pmatrix}
		0 & 0 & \ldots& 0 & 0& 0& f_1^{a-1} \\
		0 & 0 & \ldots & 0 & 0& f_1^{a-2}f_3& f_1^{a-2} f_4 \\
		0 & 0 &\ldots & 0 & f_1^{a-3}f_3^2& f_1^{a-3}f_3f_4 & f_1^{a-3}f_4^2 \\
		\vdots & \vdots & \reflectbox{$\ddots$}  & \vdots & \vdots & \vdots & \vdots \\
		0 & f_3^{a-1}& \ldots& f_3^{a-4}f_4^3  & f_3^{a-3}f_4^2& f_3^{a-2}f_4 & f_4^{a-1}
	\end{pmatrix}
	\mathcal{B}_{2,2} = 
	\begin{pmatrix}
		&   &  & f_2 &  -f_3\\
		&   & f_2&  -f_3&  \\
		&       \reflectbox{$\ddots$}    &      \reflectbox{$\ddots$}   & &\\
		f_2&   -f_3     &       &      & 
	\end{pmatrix} 
	\]
	With these notation we perform the following $2a \times 2(a+1)$-matrix $\mathcal{A}$ and the $2(a+1)\times 2a$-matrix $\mathcal{B}$ 
	\[
	\mathcal{A} = \begin{pmatrix}
		\mathcal{A}_{1,1} & \mathcal{A}_{1,2} \\
		\mathcal{A}_{2,1} & \mathcal{A}_{2,2}
	\end{pmatrix} \text{ and }
	\mathcal{B} = 
	\begin{pmatrix}
		\mathcal{B}_{1,1} & \mathcal{B}_{1,2}\\
		\mathcal{B}_{2,1} & \mathcal{B}_{2,2}
	\end{pmatrix}^T .
	\]
Moreover, we  introduce the  $2a \times (a+1)$-matrices $\mathcal{A}', \mathcal{A}''$ and  the 
$(a+1)\times 2a$-matrices $\mathcal{B}', \mathcal{B}''$ as follows
\[
\mathcal{A}' = \begin{pmatrix}
	\mathcal{A}_{1,1} \\
	\mathcal{A}_{2,1}
\end{pmatrix}, 
\mathcal{A}'' = \begin{pmatrix}
	\mathcal{A}_{1,2} \\
	\mathcal{A}_{2,2}
\end{pmatrix}, 
\mathcal{B}' = 
\begin{pmatrix}
	\mathcal{B}_{1,2} \\
	\mathcal{B}_{2,2}
\end{pmatrix}^T
\text{ and }
\mathcal{B}'' = 
\begin{pmatrix}
	\mathcal{B}_{1,1} \\
	\mathcal{B}_{2,1}
\end{pmatrix}^T.
\]
\end{construction}

\begin{construction} \label{cons-2} (B): In the next step we focus the attention of the following two $(a-1)\times 2a$-matrices:
	\[
	\mathcal{C}= 
	\begin{pmatrix}
		f_4& -f_3 &  &  &  &   &  & -f_2 &  f_1\\
		& f_4 & -f_3&  &  &   & -f_2&  f_1&  \\
		&         &     \ddots  &  \ddots &&       \reflectbox{$\ddots$}    &      \reflectbox{$\ddots$}   & &\\
		&          &       &              f_4 & -f_3\,-f_2&    f_1     &       &      & 
	\end{pmatrix}
	\]
	\[
	\mathcal{D}^T = 
	\begin{pmatrix}
		-f_2& f_3 &  &  &  &   &  & f_4 &  -f_1\\
		& -f_2 & f_3&  &  &   & f_4&  -f_1&  \\
		&         &     \ddots  &  \ddots &&       \reflectbox{$\ddots$}    &      \reflectbox{$\ddots$}   & &\\
		&          &       &              -f_2 & f_3 \quad f_4&    -f_1     &       &      & 
	\end{pmatrix}.
	\]
\end{construction}

In the sequel we will need the following substitution. 

\begin{remark} \label{rem-1}
	In case of the substitution $(f_1,f_2,f_3,f_4) \mapsto (-f_1,-f_4,-f_3,-f_2)$ it follows that 
	\[
	\mathcal{C} \cong \mathcal{D}^T, \mathcal{A}_{1,1} \cong \mathcal{B}_{1,2}  \mbox{ and }  \mathcal{A}_{2,1} \cong \mathcal{B}_{2,2}
	\] 
	and therefore also $\mathcal{A} \cong \mathcal{B}^T$.
\end{remark}

For our investigations we fix the following convention. 

\begin{notation} \label{not-1}
	During the paper we consider vectors as row vectors. Multiplication by matrices is from the right. 
\end{notation}

\section{Annihilation}
With these matrices we have the following property:

\begin{lemma} \label{lem-1}
	It holds $\mathcal{A} \cdot \mathcal{B} = 0$ and also   $\mathcal{A}' \cdot \mathcal{B}' = 0$.
\end{lemma}

\begin{proof}
	 To this end let $\mathfrak{a}_i$ denote the $i$-th row of $\mathcal{A}$ and 
	$\mathfrak{b}_j$ the $j$-th column of $\mathcal{B}$. First let $1 \leq i, j \leq a$. Then 
	\[
	\mathfrak{a}_i = (\underbrace{0, \ldots, 0,f_1, -f_2, 0, \ldots 0}_{a+1}, \underbrace{f_3^{a-i}f_4^{i-1}, f_1f_3^{a-i}f_4^{i-2}, f_1^2f_3^{a-i}f_4^{i-3}, \ldots, 
	f_1^{i-1}f_3^{a-i},0, \ldots, 0}_{a+1})
	\] 
	\[
	\mathfrak{b}_j = (\underbrace{f_2^{j-1}f_3^{a-j},f_1f_2^{j-2}f_3^{a-j}, f_1^2f_2^{j-3}f_3^{a-j}, \ldots,f_1^{j-1}f_3^{a-j}, 0, \ldots,0}_{a+1},\underbrace{0, \ldots,0,-f_1,f_4,0, \ldots, 0}_{a+1})^T.
	\]
	Then $\mathfrak{a}_i  \mathfrak{b}_j ^T = 0$. The vanishing for the remaining pairs of indices follows in a corresponding way. 
\end{proof}

In the next step we focus the attention of the following two $(a-1)\times 2a$-matrices 
$\mathcal{C}$ and $\mathcal{D}^T$ as defined above. 

\begin{lemma} \label{lem-2}
	We claim that $\mathcal{C} \cdot \mathcal{A} = 0$ and $\mathcal{B} \cdot \mathcal{D} = 0$. 
\end{lemma}

\begin{proof}
	We start with the first product. With the exception of the 
	product of the $a$-th row $\mathfrak{c}$ of $\mathcal{C}$ with the $(2a+1)$-th column $\mathfrak{a}$  of $\mathcal{A}$ all zeros 
	in the product matrix occur as Koszul relations in the corresponding separate blocks.  Moreover 
	\[
	\mathfrak{c} = (\underbrace{0, \ldots , 0, f_4, -f_3}_a, \underbrace{-f_2,f_1,0, \ldots, 0}_a) \text{ and }
	\mathfrak{a}^T = (\underbrace{0,\ldots, 0, f_1^{a-1}}_a, \underbrace{0,f_1^{a-2}f_3, %f_1^{a-3}f_2f_3,
		\ldots, f_2^{a-2}f_3}_a)
	\]
	and therefore $\mathfrak{c} \mathfrak{a}^T = 0$. The vanishing of $\mathcal{B} \mathcal{D} = 0$ follows by symmetry. 
\end{proof}

\section{The complexes}
 As shown in Lemma \ref{lem-2} we have that 
 $$
 \mathcal{C} \cdot \mathcal{A}' =  \mathcal{C} \cdot \mathcal{A}'' = 0 \mbox{ and }\;  \mathcal{B}' \cdot \mathcal{D} 
 = \mathcal{B}'' \cdot \mathcal{D} = 0.
 $$ 
 This lead to the following four complexes: 
 
 \begin{theorem} \label{res-1}
 	With the previous notation  $M := \Coker \mathcal{A}'$ and $N := \Coker \mathcal{D}$. Then we have  the following two complexes
 	\[
 	C^{\bullet}_1: 0 \to A^{a-1} \stackrel{\mathcal{C}}{\longrightarrow} A^{2a} \stackrel{\mathcal{A}'}{\longrightarrow} A^{a+1}\;  \text{ and } \quad
 	C^{\bullet}_2: 0 \to A^{a+1} \stackrel{\mathcal{B}'}{\longrightarrow} A^{2a} \stackrel{\mathcal{D}}{\longrightarrow} A^{a-1} 
 	\]
 	\begin{itemize}
 		\item[(a)] $C_1^{\bullet}$ is exact if and only if  $\grade (f_1,f_2,f_3,f_4;A) \geq 2$ and $f_1f_4-f_2f_3$ is a 
 		non-zero divisor.
 		\item[(b)] $C^{\bullet}_2 \cong \Hom_A(C^{\bullet}_1,A)$ and $H^2(C^{\bullet}_2) \cong N$. 
 		\item[(c)] If $C^{\bullet}_1$ is exact, then 
 		$H^i(C^{\bullet}_2) \cong \Ext^i_A(M,A)$ and $\Ext^2_A(M,A) \cong N$.
 	\end{itemize}
 \end{theorem}
 
 \begin{proof} For the proof of (a) it is easily seen that 
 	$$
 	 I_{a-1}(\mathcal{C}) = (f_1,f_2,f_3,f_4)^{a-1}  \text{ and }  I_{a+1}(\mathcal{A}') = (f_1,f_2,f_3,f_4)^{a-1}(f_1f_4-f_2f_3).
 	$$
 	By the Buchbaum-Eisenbud acyclicity criterion (see e.g. \cite{Ed}) it follows that $	C^{\bullet}_1$ is exact 
 	if and only if  $\grade (f_1,f_2,f_3,f_4;A) \geq 2$ and $f_1f_4-f_2f_3$ is a 
 	non-zero divisor.
 	
 	For the proof of (b) first recall that $\mathcal{C}^T \cong \mathcal{D}$ and $(\mathcal{A}')^T \cong \mathcal{B}'$ as easily seen 
 	(see \ref{rem-1}).  Then the proof of the statement (c) follows by definitions. 
  \end{proof}
  
  There is a similar construction of complexes as in \ref{res-1}.
  
 \begin{theorem} \label{res-2}
 	With the previous notation  et $M':= \Coker \mathcal{A}''$ and $N := \Coker \mathcal{D}$ as above. Then we have  the following two complexes
 	\[
 	D^{\bullet}_1: 0 \to A^{a-1} \stackrel{\mathcal{C}}{\longrightarrow} A^{2a} \stackrel{\mathcal{A}''}{\longrightarrow} A^{a+1}\;  \text{ and } \quad
 	D^{\bullet}_2: 0 \to A^{a+1} \stackrel{\mathcal{B}''}{\longrightarrow} A^{2a} \stackrel{\mathcal{D}}{\longrightarrow} A^{a-1}. 
 	\]
 	\begin{itemize}
 		\item[(a)] $D_1^{\bullet}$ is exact if and only if  $f_1f_3$ is a non zero divisor.. 
 		\item[(b)]  $D^{\bullet}_2 \cong \Hom_A(D^{\bullet}_1,A)$ and $H^2(D^{\bullet}_2 )\cong N$.
 		\item[(c)] If $D^{\bullet}_1$ is exact, then 
 		$H^i(D^{\bullet}_2) \cong \Ext^i_A(M',A) $ and $ \Ext_A^2(M',A) \cong N$. 
 	\end{itemize}
 \end{theorem}
 
 \begin{proof}
 	Clearly $\mathcal{C} \cdot \mathcal{A}'' = \mathcal{B}'' \cdot \mathcal{D} = 0$. This provides the complexes 
 	$D^{\bullet}_1$ and $D^{\bullet}_2$. 
 	As shown before  $I_{a-1}(\mathcal{C}) = I_{a-1}(\mathcal{D}) =  (f_1,f_2,f_3,f_4)^{a-1}$. By similar arguments 
 	it follows that $\Rad I_{a+1} (\mathcal{A}'') = \Rad I_{a+1} (\mathcal{B}'') =\Rad (f_1f_3)$ and (a) and (b) follow. The statements in (c) are consequences of the definitions.
 \end{proof}
  
  With the previous matrices there is another complex.
 
\begin{theorem} \label{res-3}
	Let $\underline{f} = f_1,f_2,f_3,f_4$ denote a sequence of four elements of a commutative ring $A$ and let $N := \Coker \mathcal{D}$. 
	Then there is a complex of free $A$-modules 
	\[
	C^{\bullet}: 0 \to A^{a-1} \stackrel{\mathcal{C}}{\longrightarrow} A^{2a}   \stackrel{\mathcal{A}}{\longrightarrow} A^{2(a+1)} 
	 \stackrel{\mathcal{B}}{\longrightarrow} A^{2a}  \stackrel{\mathcal{D}}{\longrightarrow} A^{a-1}.
	\]
	\begin{itemize}
		\item[(a)] $C^{\bullet}$ is exact and therefore a free resolution of $N$ if and only if $\underline{f}$ is a regular sequence. 
		\item[(b)] $C^{\bullet} \cong \Hom_A(C^{\bullet},A)$.  If $C^{\bullet}$ is exact, then $\Ext^i_A(N,A) = 0$ for $i \not= 4$ 
		and $\Ext^4_A(N,A) \cong N.$
	\end{itemize}
\end{theorem}

\begin{proof}
	For the proof of (a) we first recall $I_{a-1}(\mathcal{C}) = I_{a-1}(\mathcal{D}) =  (f_1,f_2,f_3,f_4)^{a-1}$ as shown in the proof 
	\ref{res-1}.  With the description of $I_{a+1}(\mathcal{B})$ and the definition of $\mathcal{B}$ it follows that $\Rad I_{a+1}(\mathcal{B}) = 
	(f_1,f_2,f_3,f_4)$. The same holds for $\Rad I_{a+1}(\mathcal{A})$. Then the Buchsbaum-Eisenbud acyclicity  criterion 
	shows that $C^{\bullet}$ is exact if and only if $\underline{f}$ forms a regular sequence. 
	
	For the proof of (b) recall that  $\mathcal{C}^T \cong \mathcal{D}$ and $\mathcal{A}^T \cong \mathcal{B}$ as seen easily (see \ref{rem-1}).
\end{proof}

With the previous result in mind we can describe the first co-homology of $C_2^\bullet$ and $D_2^{\bullet}$ in more detail.

\begin{corollary} \label{cor-1}
	Suppose that $\underline{f}$ is a regular sequence. Then $H_0(C_2^\bullet) = H_0(D_2^{\bullet}) = 0$ and 
	\[
	\Ext^1_A(M,A) \cong H_1(C_2^{\bullet}) \cong \im \mathcal{B}'' \text{ and } 	
	\Ext_A^1(M',A)\cong H_1(D_2^{\bullet}) \cong \im \mathcal{B}'.
	\]
\end{corollary}

\begin{proof}
	First the vanishing is clear. For the proof of the rest observe that $\ker \mathcal{D} = \im \mathcal{B}$ as follows by 
	Theorem \ref{res-3}. Moreover, $\im \mathcal{B} = \im \mathcal{B}' + \im \mathcal{B}''$, which concludes the argument. 
\end{proof}

%{\color{blue} 
Before we shall give a particular application to monomial curves let us consider a general feature. 
Let $A$ denote a commutative Noetherian ring and let $A \to B$ be a ring homomorphism with $B$ 
a commutative Noetherian  ring $B$. For a given $A$-regular sequence $\underline{f} = f_1,f_2,f_3,f_4$ and an integer $a \geq 2$ 
we form the matrix $\mathcal{D}$ as above. 

\begin{remark} \label{rem}
	By view of \ref{res-3} we get the resolution 
	\[
	C^{\bullet}: 0 \to A^{a-1} \stackrel{\mathcal{C}}{\longrightarrow} A^{2a}   \stackrel{\mathcal{A}}{\longrightarrow} A^{2(a+1)} 
	\stackrel{\mathcal{B}}{\longrightarrow} A^{2a}  \stackrel{\mathcal{D}}{\longrightarrow} A^{a-1}.
	\]
	Now suppose that $\underline{f}$ is also a $B$-regular sequence. Tensoring by $B$ yields an exact sequence 
	\[
		C^{\bullet}\otimes_A B: 0 \to B^{a-1} \stackrel{\mathcal{C}}{\longrightarrow} B^{2a}   \stackrel{\mathcal{A}}{\longrightarrow} B^{2(a+1)} 
	\stackrel{\mathcal{B}}{\longrightarrow} B^{2a}  \stackrel{\mathcal{D}}{\longrightarrow} B^{a-1}.
	\]
	That is, $\Tor_i^A(\Coker \mathcal{D},B) = 0$ for $i > 0$.
\end{remark}

A concrete example is the following.

\begin{example} \label{ex-2}
	 Let $A$ be  the polynomial ring over the field $\Bbbk$ in the variables $x_1,x_2,x_3,y_1,y_2,y_3$. 
	 We consider the ideal $I \subset A$ generated by the $2 \times 2$-minors of the matrix 
	 $\begin{pmatrix}
	 	x_1 &x_2 & x_3\\y_1&y_2 & y_3. 
	 \end{pmatrix}$
	 We define $B = A/I$. Then $x_1,x_2-y_1,x_3-y_2,y_3$ is an $A$-regular as well as  a $B$-regular sequence. 
\end{example}

%}

\section{An application}
Let 
${C} \subset \mathbb{P}^3_{\Bbbk}$ denote an integral non-degenerate curve in 
projective $3$-space $\mathbb{P}^3_{\Bbbk}$  over a field $\Bbbk$ with $d = \deg {C}$. 
Let $A = \Bbbk[x_0,x_1,x_2,x_3]$ the polynomial ring in $4$-variables and let $I_C \subset A$ 
denote the defining ideal of ${C}$. We are interested in the following monomial curves:

\begin{example} \label{ex-1}
	Let $d \geq 3$ and let ${C} \subset \mathbb{P}^3_\Bbbk$ be the curve given as the image of 
	\[
	\mathbb{P}^1_\Bbbk \to \mathbb{P}^3_\Bbbk,   \; (s:t) \mapsto (s^d:s^at^b:s^bt^a:t^d) 
	\text{ with } 0 < b < a, d = a+b, \gcd(a,b) = 1.
	\]
	Its defining 
	ideal $I_C \subset A$ is given by $I_C = (Q, F_i, 0 \leq i \leq a-b)$ 
	with 
	\[
	Q = x_0x_3-x_1x_2, \text{ and } F_i = x_0^{a-b-i}x_2^{b+i} -x_1^{a-i}x_3^i,\;  i = 0,\ldots,a-b
	\]
	(see e.g. \cite{BR} or \cite{BSV}). Note that   $Q = x_0x_3-x_1x_2$ defines a smooth quadric. Moreover, it follows $p_a({C}) = (a-1)(b-1)$ for the arithmetic genus. In the following we assume that $A/I_C$ is not a Cohen-Macaulay ring. 
	This holds if and only if $a-b > 2$, supposed in the sequel. 
\end{example}

Next we summarize a few - more or less - well known properties of these monomial curves. For the definition and basic facts about 
the Castelnuovo-Mumford regularity we refer to \cite{Ed}. 

\begin{observation} \label{obs-1}
	(A) We have $\reg A/I_C = a$ (see e.g. \cite{LP} or \cite{Sp1}). By the short exact sequence 
	\[
	0 \to I_C/(Q) \to A/(Q) \to A/I_C \to 0
	\]
	it follows that $M: = I_C/(Q) $ is a graded $A$-module with $\dim M = 3, \depth M = 2$ and $\reg M = a$. To this end 
	note that $A/(Q)$ is a $3$-dimensional Cohen-Macaulay ring with $\reg A/(Q) = 2$.\\
	(B) By view of \ref{ex-1} it follows $a(M) = a$ for the initial degree of $M = I_C/(Q)$. Because of $\reg M = a$ we have a 
	linear minimal free resolution of $M$ (see e.g. \cite{LP} or  \cite{SP2}). By the previous results there is an explicit description 
	of the minimal free resolution. 
\end{observation}

\begin{theorem} \label{thm-3}
	With the previous notation we substitute $a$ by $a-b+1$ and $(f_1,f_2,f_3,f_4)$ by the variables $(x_3,x_1,x_0,x_2)$ in the 
	constructions \ref{cons-1} and \ref{cons-2}. Then 
	\begin{itemize}
		\item[(a)] $M \cong A^{a-b+1}(-a) /\im \mathcal{A}'$ and 
		\item[(b)] $ 0 \to A^{a-b-1} (-a-2)\stackrel{\mathcal{C}}{\longrightarrow} A^{2(a-b)}(-a-1) 
		\stackrel{\mathcal{A}'}{\longrightarrow} A^{a-b+1}(-a)$ is a graded minimal free resolution of $I_C/(Q)$.
	\end{itemize}
\end{theorem}

\begin{proof}
	With the previous notation there are the following relations 
	\[
		x_2F_i -x_0F_{i+1} = x_1^{a-i-1}x_3^i Q\; \text{ and } x_3F_i -x_1F_{i+1} = x_0^{a-b-i-1}x_2^{b+i} Q \; \text{ for } i= 0,\ldots, a-b-1. 
	\]
	Then we define a free presentation $A^{a-b+1}(-a) \to M$ by 
	\[
	(0,\ldots,0,1,0,\ldots,0) \mapsto F_i +(Q)/(Q) \; \text{ for } i = 0, \ldots, a-b.
	\]
	Next we introduce the two $(a-b)\times(a-b+1)$-matrices 
	\[
		\mathcal{A}_{1,1}= 
	\begin{pmatrix}
		x_3 & -x_1&  &  &  \\
		& x_3 & -x_1&  &  \\
		&         &     \ddots  &  \ddots &\\
		&          &       &              x_3& -x_1
	\end{pmatrix} 
	\text{ and }
		\mathcal{A}_{2,1} = 
	\begin{pmatrix}
		&   &  & -x_2 &  x_0\\
		&   & -x_2&  x_0&  \\
		&       \reflectbox{$\ddots$}    &      \reflectbox{$\ddots$}   & &\\
		-x_2&    x_0     &       &      & 
	\end{pmatrix} 
	\]
	That is we obtain these matrices by substituting $(f_1,f_2,f_3,f_4)$ by the variables $(x_3,x_1,x_0,x_2)$ in the 
	constructions \ref{cons-1} and \ref{cons-2}. By view of Theorem \ref{res-1} it follows that 
	$M \cong \Coker \mathcal{A}'$ and $C_1^{\bullet}$ provides the minimal graded resolution of $M$. 
\end{proof}

For the following application we need the notion of the canonical module of a finitely  generated module 
$M$. Here we note it by $\omega(M)$. For several of the technical details we refer to the book \cite{Sp3}. 

\begin{corollary} \label{cor-2}
	With the notation of \ref{thm-3} and the constructions of  \ref{cons-1} and \ref{cons-2} we have: 
	\begin{itemize}
		\item[(a)] $\omega(M) \cong M(2)$.
		\item[(b)] $0 \to A^{a-b}(-a) \stackrel{\mathcal{C}}{\longrightarrow} A^{2(a-b+1)}(-a+1) 
		\stackrel{\mathcal{A}''}{\longrightarrow} A^{a-b+2}(-a+2)$ is a graded minimal free resolution of $\omega(M)$. 
 	\end{itemize}
\end{corollary}

\begin{proof}
	The short exact sequence $0 \to I_C/(Q) \to A/(Q) \to A/I_C \to 0$ induces a short exact sequence 
	of canonical modules 
	\[
	0 \to \omega(A//Q) \to \omega(M) \to \omega(A/I_C) \to 0.
	\]
	By view of the results in \cite{Sp1} it follows that the initial degree of $\omega(M)$ is equal to $a-2$. 
	By view of \ref{cor-1} it follows that $\omega(M)(-4) \cong \im \mathcal{A}'$. Then both of the statements 
	follow. 
\end{proof}

Next we shall investigate the Hartshorne-Rao module for the monomial curves as above. Note that 
it is defined by $HR(C) = H^1_{A_+}(A/I_C)$ (see \cite{Hr2} and \cite{Rp}). 

\begin{corollary} \label{cor-3}
	With the previous notation the following results hold:
	\begin{itemize}
		\item[(a)] $HR(C) (b)\cong A^{a-b-1}/\im \mathcal{D}$.
		\item[(b)] 
	   The minimal free resolution of $HR(C)$ is given by
		\begin{gather*}
		0 \to A^{a-b-1}(-a-2) \stackrel{\mathcal{C}}{\longrightarrow} A^{2(a-b)}(-a-1)  \stackrel{\mathcal{A}}{\longrightarrow} A^{a-b+1}(-a) \oplus A^{a-b+1}(-b-2) \\ \stackrel{\mathcal{B}}{\longrightarrow} 
		A^{2(a-b)}(-b-1)  \stackrel{\mathcal{D}}{\longrightarrow} A^{a-b-1}(-b).
		\end{gather*}
		\item[(c)] $\Hom_\Bbbk(HR(C), \Bbbk) \cong HR(C)(a+b-2)$.
		\item[(d)] $\dim_\Bbbk HR(C)_j = (j-b+1)(a-j-1)$ if $ b \leq j \leq a-2$ and zero else.
		\item[(e)] $\ell_A(HR(C)) = \tbinom{a-b+1}{3}$.
	\end{itemize}
\end{corollary}

\begin{proof}
	(a): The statement in (a) follows by \ref{res-1} since $\Hom_\Bbbk(H^2_{A_+}(M),\Bbbk) \cong \Ext_A^2(M,A(-4)) \cong N(-4)$. Moreover, notice $H^2_{A_+}(M) \cong H^1_{A_+}(A/I_C)$. Now recall the isomorphism in \ref{res-3} (b), so that 
	the claim follows by local duality. \\
	(b), (c): Both statements are a consequence of Theorem \ref{res-3}. \\
	(d), (e): To this end count the Hilbert functions in order to obtain (d) and sum up for (e). 
\end{proof}

Some computational experiments are done by the aid of {\sc{Singular}} (see \cite{DGPS}).

\medskip

{\bf{Acknowledgment}:}
The present paper was partially written while the first author stayed at Max Planck Institute for Mathematics in the Science, Leipzig, Germany, August 15--September 8, 2025.

%\bibliographystyle{siam}
%\bibliographystyle{siam}
%\bibliographystyle{acm}
%\bibliographystyle{abbrvdin}
%\bibliographystyle{alphadin}
%\bibliography{curve}

\end{document}